\documentclass[oneside,english]{article}
\usepackage[T1]{fontenc}
\usepackage[latin9]{inputenc}
\usepackage{float}
\usepackage{epic,eepic}
\usepackage{comment}
\usepackage{amstext,amsthm,amssymb,amsmath}
\usepackage{epsfig,subfigure,epstopdf}
\usepackage{graphicx}
\usepackage{subfigure}
\usepackage{wrapfig}
\usepackage{fullpage}
\usepackage{color}
\usepackage{multirow}
\usepackage{tabulary}
\usepackage{booktabs}
\usepackage{enumerate}
\usepackage{color}

\usepackage[title]{appendix}

\theoremstyle{definition}
\newtheorem{theorem}{Theorem}[section]
\newtheorem{example}{Example}
\newtheorem{lemma}{Lemma}[section]

\makeatletter
\numberwithin{equation}{section}
\numberwithin{figure}{section}

\makeatother

\title{Hybrid explicit-implicit learning for multiscale problems with time dependent source}
\author{Yalchin Efendiev\footnote{Department of Mathematics, Texas A\&M University, College Station, TX 77843, USA}, \ Wing Tat Leung\footnote{Department of Mathematics, City University of Hong Kong, Hong Kong SAR}, \ Wenyuan Li\footnote{Department of Mathematics, Texas A\&M University, College Station, TX 77843, USA}, \ Zecheng Zhang\footnote{Department of Mathematical Sciences,
Carnegie Mellon University,
Pittsburgh, PA 15213, USA}}

\usepackage{babel}
\begin{document}

\maketitle

\section*{Abstract}
The splitting method is a powerful method for solving partial differential 
equations. 
Various splitting methods have been designed to separate different physics, 
nonlinearities, and so on.
Recently, a new splitting approach has been proposed where 
some degrees of freedom are
handled implicitly while other degrees of freedom are handled explicitly. 
As a result, the scheme contains two equations, one implicit and the other explicit.
The stability of this approach has
been studied. It was shown that the time step scales
 as the coarse spatial mesh size, which
can provide a significant computational advantage. However, 
the implicit solution part can be expensive,
especially for nonlinear problems. In this paper, we introduce modified
 partial machine learning algorithms to 
replace the implicit solution part of the algorithm. 
These algorithms are first introduced in 
\textit{`HEI: Hybrid Explicit-Implicit Learning For Multiscale Problems'}, 
where
a homogeneous source term is considered along with the Transformer, 
which is a neural network that can 
predict future dynamics. In this paper, we consider time-dependent source terms which is a generalization of the previous work. Moreover, we use the whole history of the solution
to train the network. As the implicit part of the equations is more complicated 
to solve, we design a neural network to predict it based on 
training. Furthermore,
 we compute the explicit part of the solution using our 
splitting strategy. In addition, we use 
Proper Orthogonal Decomposition based model reduction in machine learning. 
The machine learning algorithms provide computational 
saving without sacrificing accuracy. We present three numerical 
examples which show that our machine learning scheme is stable and accurate.

\section{Introduction}

Many physical problems vary at multiple space and time scales. The solutions of these problems are difficult
to compute as we need to resolve the time and space scales. For example, porous media flow
and transport occur over many space and time scales. Many multiscale and splitting algorithms
have been proposed to solve these problems \cite{chung_partial_expliict21,chung2021contrastindependent,li2021partially}. In this paper, we would like to combine machine learning and splitting/multiscale approaches
to predict the dynamics of complex forward problems. Our goal is to perform machine learning
on a small part of the solution space, while computing ``the rest'' of the solution fast. Our proposed
splitting algorithms allow doing it, which we will describe below.

Our approaches use multiscale methods in their construction to achieve coarser time stepping.
Multiscale methods have extensively been studied in the literature.
For linear problems, homogenization-based approaches 
\cite{eh09,le2014msfem, leung2021nh}, multiscale
finite element methods \cite{eh09,hw97,jennylt03}, 
generalized multiscale finite element methods (GMsFEM) 
\cite{chung2016adaptiveJCP,MixedGMsFEM,WaveGMsFEM,chung2018fast,GMsFEM13}, 
Constraint Energy Minimizing GMsFEM (CEM-GMsFEM) 
\cite{chung2018constraint, chung2018constraintmixed, chetverushkin2021computational}, nonlocal
multi-continua (NLMC) approaches \cite{NLMC},
metric-based upscaling \cite{oz06_1}, heterogeneous multiscale method 
\cite{ee03}, localized orthogonal decomposition (LOD) 
\cite{henning2012localized}, equation-free approaches \cite{rk07,skr06}, 
multiscale stochastic approaches \cite{hou2017exploring, hou2019model, hou2018adaptive},
and hierarchical multiscale methods \cite{brown2013efficient} have been studied. 
Approaches such as GMsFEM and NLMC are designed to handle high-contrast problems.
For nonlinear problems, we can replace linear multiscale basis functions with nonlinear maps \cite{ep03a,ep03c,efendiev2014generalized}.

Splitting methods are introduced when solving dynamic problems \cite{ascher1997implicit,shi2019local,frank1997stability}. There is a wide range of applications for splitting methods, e.g. splitting physics, scales, and so on. Splitting methods allow simplifying the large-scale problem into smaller-scale problems. We couple them  and obtain the final solution after we solve the smaller-scale problems. Recently, a partially explicit splitting method \cite{chung_partial_expliict21,chung2021contrastindependent, https://doi.org/10.48550/arxiv.2204.00554} is proposed for linear and nonlinear multiscale problems. In this approach, we treat some solution parts implicitly while treating the other parts explicitly. The stability analysis shows that the time step only depends on the explicit part of the solution. With the careful splitting of the space, we are able to make the time step scale as the coarse mesh size and independent of contrast.

To illustrate the main idea of our methods, we consider
\[  u_t + f(u)=g(x,t).  \]
In this paper, we combine machine learning and splitting approaches (such as the partially explicit method).
In the splitting algorithm, we write the solution as $u=u_1+u_2$, where $u_1$ is computed implicitly and $u_2$ is computed explicitly. 
In our previous work, 
the Transformer has been used to solve this kind of problem in \cite{efendiev2021hei, chung2021multi}. In \cite{efendiev2021hei}, 
we consider zero source term and use the Transformer to predict the future dynamics of the solution. In our current paper, we consider full dynamics and time-dependent source terms. The latter occurs in porous media applications and represents the rates of wells, which vary in time. The goal in these simulations is to compute the solutions at all time steps for a given time-dependent source term in an efficient way, which is not addressed before.

Many machine learning and multiscale methods are designed for solving multiscale parametric PDEs \cite{chung2021multi, 10.2118/173271-PA, zhang2020learning, chung2020multi, lin2021multi, lin2021accelerated}. 
The advantage of machine learning is that it treats  PDEs as an input-to-output map. By training, and possibly, using the observed data, machine learning avoids computing solutions for each input (such as parameters in boundary conditions, permeability fields or sources). In this paper, we use machine learning to learn the difficult-to-compute part of the solution that is, in general, computed implicitly. This is a coarse-grid part of the solution, which can also be observed given
some dynamic data. The splitting algorithm gives two coupled equations. The first equation for $u_1$ is implicit while the second equation for $u_2$ is explicit. We design a neural network to predict the first component of the solution. Then, we calculate the second component using the discretized equations. The inputs of our neural network are the parameters that characterize the source term of the equation. The outputs are the whole history of $u_1$, i.e. $u_1$ at all time steps. The output dimension of our neural network is the dimension of $u_1$ times the number of time steps, which is considerably large. Therefore, we use Proper Orthogonal Decomposition (POD) \cite{10.2118/173271-PA, chung2021computational, Kunisch2001GalerkinPO} to reduce the dimension of the solution space. We concatenate all  solution vectors of every training input together to generate a large-scale matrix. After doing the singular value decomposition, we select several dominant eigenvectors to form the POD basis vectors. Thus, the output becomes the coordinate vectors based on the POD basis vectors and its dimension is greatly reduced, so is the number of parameters inside the neural networks.

In this paper, we show several numerical examples to illustrate our learning framework which is based on the partially explicit scheme. In our numerical examples, the diffusion term $f(u)$ is $-div(\kappa(x,u) \nabla u)$ and the reaction term is $g(x,t)$. $g(x,t)$ is a source term that is nonzero on several coarse blocks and is zero elsewhere. We choose such $g(x,t)$ as it simulates the wells in petroleum applications. By comparing with the solution which is obtained using fine grid basis functions, we find that the relative $L^2$ error is small and similar to the scheme where we calculate both components. Thus, the learning scheme can achieve similar accuracy while saving computation.

This paper is organized as follows. We present preliminaries and review the partially explicit splitting scheme in Section \ref{secPS}. In Section \ref{secML}, we show details about our learning framework. We present numerical experiments in Section \ref{secNR}.







\section{Problem Setting}
\label{secPS}
We consider the following parabolic partial differential equation
\begin{align}
    u_t + f(u) = g(x,t).
    \label{eq2_1}
\end{align}
Let $V$ be a Hilbert space. In this equation, $f = \cfrac{\delta F}{\delta u} \in V^{\ast}$ with $F$ being the variational derivative of energy $F(u) := \int_{\Omega} E(u)$. In this paper, we use $(\cdot , \cdot)$ to denote $(\cdot,\cdot)_{V^*,V}$ for simplicity. We assume that $f$ is contrast dependent (linear or nonlinear). Here, $g(x,t)$ is a time dependent source term with parameters. 
The weak formulation of Equation (\ref{eq2_1}) is the following. We want to find $u\in V$ such that
\begin{align}
    (u_t, v) + (f(u), v) = (g(x,t), v) \quad \forall v \in V.
    \label{weakfor0}
\end{align}
The $(\cdot,\cdot)$ is the $L^2$ inner product.
We assume that 
\begin{align}
\label{assumptionf}
(f(v),v) > 0,\quad \forall v \in V.
\end{align}
\begin{example}
In the heat equation, $F(u^{\prime}) = \cfrac{1}{2} \int_{\Omega} \kappa |\nabla u^{\prime}|^2$. Let $V = H^1(\Omega)$. We have $f(u) = \cfrac{\delta F}{\delta u^{\prime}}(u) \in (H^1(\Omega))^* = H^1(\Omega)$.
Then,
\[ (\cfrac{\delta F}{\delta u^{\prime}} (u),v) = \int_{\Omega} \kappa \nabla u \cdot \nabla v.  \]
Therefore, we have
\[ f(u) = -\nabla \cdot (\kappa \nabla u).   \]
The weak formulation (\ref{weakfor0}) becomes 
\[ (u_t , v) - (\nabla \cdot (\kappa \nabla u) , v) = (g(x,t),v).  \]
\end{example}

The standard approach to solving this problem is the finite element method. 
Let $V_H$ be the finite element space and $u_H \in V_H$ is the numerical solution satisfying
\begin{align}
    (u_{H,t}, v) + (f(u_H), v) = (g(x,t), v) \quad \forall v \in V_H.
    \label{refsol}
\end{align}
The Backward Euler (implicit) time discretization for this Equation is 
\begin{align}
    (\frac{u^{n+1}_{H}- u^n_H}{\Delta t } , v) + ( f(u^{n+1}_H) , v) = 
    (g(x,t^{n+1}) , v) \quad \forall v \in V_H.
    \label{refsol1}
\end{align}
The Forward Euler (explicit) time discretization for Equation (\ref{refsol}) is 
\begin{align}
    (\frac{u^{n+1}_{H}- u^n_H}{\Delta t } , v) + ( f(u^{n}_H) , v) = 
    (g(x,t^{n}) , v) \quad \forall v \in V_H.
\end{align}
One can find details about the stability analysis of the Backward and Forward Euler method in \cite{chung2021contrastindependent}.

\subsection{Partially Explicit Scheme}
In this paper, we use the partially explicit splitting scheme introduced in \cite{chung2021contrastindependent}.
We split the space $V_H$ to be the direct sum of two subspaces $V_{H,1}$ and $V_{H,2}$, i.e. $V_H = V_{H,1} \oplus V_{H,2}$.
Then the solution $u_H=u_{H,1}+u_{H,2}$ with $u_{H,1}\in V_{H,1}$ and $u_{H,2}\in V_{H,2}$ satisfies
\begin{align*}
    ((u_{H,1}+u_{H,2})_t, v_1) + (f(u_{H,1}+u_{H,2}), v_1) = (g(x,t), v_1) \quad \forall v_1 \in V_{H,1},\\
    ((u_{H,1}+u_{H,2})_t, v_2) + (f(u_{H,1}+u_{H,2}), v_2) = (g(x,t), v_2) \quad \forall v_2 \in V_{H,2}.
\end{align*}
We use the partially explicit time discretization for this problem.
Namely, we consider finding $\{u_H^n\}_{n=0}^K$ (with $u_H^n = u_{H,1}^n + u_{H,2}^n,\; \forall n$) such that
\begin{align}
&(\frac{u_{H, 1}^{n+1}-u_{H, 1}^{n}}{\Delta t}+\frac{u_{H, 2}^{n}-u_{H, 2}^{n-1}}{\Delta t}, v_{1})+(f(u_{H, 1}^{n+1}+u_{H, 2}^{n}),v_1) = (g(x,t^n), v_{1}) \quad\forall v_{1} \in V_{H, 1},
\label{pe1}\\
&(\frac{u_{H, 1}^{n}-u_{H, 1}^{n-1}}{\Delta t}+\frac{u_{H, 2}^{n+1}-u_{H, 2}^{n}}{\Delta t}, v_{2})+(f(u_{H, 1}^{n+1}+u_{H, 2}^{n}),v_2) = (g(x,t^n), v_{2}) \quad\forall v_{2} \in V_{H, 2} .
\label{pe2}
\end{align}
We refer to \cite{chung_partial_expliict21,chung2021contrastindependent} for the details of convergence assumptions and stability analysis of the above scheme.

\subsubsection{Construction of $V_{H,1}$ and $V_{H,2}$}
\label{SpaceConstructSection}

In this section, we present how to construct the two subspaces for our partially explicit splitting scheme.
We will first present the CEM method which is used to construct $V_{H,1}$.
Next, with the application of eigenvalue problems, we build $V_{H,2}$.
In the following discussion, let $S \subset \Omega$ and then we define $V(S) = H_0^1(S)$.

\textbf{CEM method}

We introduce the CEM finite element method in this section.
Let $\mathcal{T}_{H}$
be a coarse grid partition of $\Omega$. For $K_{i}\in\mathcal{T}_{H}$,
we need to construct a collection of auxiliary basis functions in $V(K_{i})$.
We assume that  $\{\chi_i\}$ are basis functions that
form a  partition of unity, e.g., piecewise linear functions,
see \cite{bm97}.
We first find $\lambda_j^{(i)}$ and $\psi_j^{(i)}$ such that the following equation holds:
\begin{align*} \int_{K_i} \kappa \nabla \psi_j^{(i)} \cdot \nabla v = \lambda_j^{(i)} s_i ( \psi_j^{(i)},v), \quad
\forall v \in V(K_i), \end{align*}
where \begin{align*} s_i(u,v) = \int_{K_i} \tilde{\kappa} u v, \quad
\tilde{\kappa} = \kappa H^{-2} \; \text{or} \; 
\tilde{\kappa} = \kappa \sum_{i}\left|\nabla \chi_{i}\right|^{2}.   \end{align*}  
With rearrangement if necessary, we select the $L_i$ eigenfunctions of the corresponding first $L_i$ smallest eigenvalues. 
We define \[ V_{aux}^{(i)}:=\text{span}\{\psi_{j}^{(i)}:\;1\leq j\leq L_{i}\}.  \]
Define a  projection operator $\Pi:L^{2}(\Omega)\mapsto V_{aux}\subset L^{2}(\Omega)$
\[
s(\Pi u,v)=s(u,v),\quad\forall v\in V_{aux}:=\sum_{i=1}^{N_{e}}V_{aux}^{(i)},
\]
with $s(u,v):=\sum_{i=1}^{N_{e}}s_{i}(u|_{K_{i}},v|_{K_{i}})$ and $N_e$ being the number of coarse elements.
We denote an oversampling domain which is several coarse blocks larger than $K_i$ as $K_{i}^+$\cite{chung2018constraint}.  
For every auxiliary basis $\psi_{j}^{(i)}$, we solve the following two equations and find $\phi_{j}^{(i)}\in V(K_{i}^{+})$ and $\mu_{j}^{(i)} \in V_{aux}$ 
\begin{align*}
a(\phi_{j}^{(i)},v)+s(\mu_{j}^{(i)},v) & =0,\quad\forall v\in V(K_{i}^{+}),\\
s(\phi_{j}^{(i)},\nu) & =s(\psi_{j}^{(i)},\nu),\quad\forall\nu\in V_{aux}(K_{i}^{+}).
\end{align*}
Then we can define the CEM finite element space as 
\begin{align*}
V_{cem} & :=\text{span}\{\phi_{j}^{(i)}:\;1\leq i\leq N_{e},1\leq j\leq L_{i}\}.
\end{align*}
We define $\tilde{V}:= \{v\in V: \; \Pi(v) = 0 \}$ and we will need it when we construct $V_{H,2}$ in the next subsection.

\textbf{Construction of $V_{H,2}$}

We build the second subspace $V_{H,2}$ based on $V_{H,1}$ and $\tilde{V}$. Similarly, we first have to construct the auxiliary basis functions by solving the eigenvalue equations. For every coarse element $K_{i}$, we search for $(\xi_{j}^{(i)},\gamma_{j}^{(i)})\in(V(K_{i})\cap\tilde{V})\times\mathbb{R}$ such that
\begin{align}
\label{eq:spectralCEM2}
\int_{K_{i}}\kappa\nabla\xi_{j}^{(i)}\cdot\nabla v & =\gamma_{j}^{(i)}\int_{K_{i}}\xi_{j}^{(i)}v, \;\ \forall v\in V(K_{i})\cap\tilde{V}.
\end{align}
We choose the $J_i$ eigenfunctions corresponding to the smallest $J_i$ eigenvalues.
The second type of auxiliary space is defined as $V_{aux,2} := \text{span}\{\xi_j^{(i)} : 1\leq i \leq N_e, 1\leq j\leq J_i \}$.
For every auxiliary basis $\xi_j^{(i)} \in V_{aux,2}$, we look for a basis function $\zeta_{j}^{(i)} \in V(K_i^+)$ such
that for some $\mu_{j}^{(i),1} \in V_{aux}$, $ \mu_{j}^{(i),2} \in V_{aux,2}$, we have 
\begin{align}
a(\zeta_{j}^{(i)},v)+s(\mu_{j}^{(i),1},v)+ ( \mu_{j}^{(i),2},v) & =0, \quad\forall v\in V(K_i^+), \label{eq:v2a} \\
s(\zeta_{j}^{(i)},\nu) & =0, \quad\forall\nu\in V_{aux}, \label{eq:v2b} \\
(\zeta_{j}^{(i)},\nu) & =( \xi_{j}^{(i)},\nu), \quad\forall\nu\in V_{aux,2}. \label{eq:v2c}
\end{align}
We define $$V_{H,2}=\text{span}\{\zeta_{j}^{(i)}| \; 1\leq i \leq N_e, \;  1 \leq j\leq J_i\}.$$

\section{Machine Learning Approach}
\label{secML}

In Equation (\ref{eq2_1}), we consider $g(x,t)$ as a source term with varying parameters representing temporal changes. For every parameter, we want to find the corresponding solution $u(x,t)$. 
For the traditional finite element methods, we need to construct meshes and basis functions, assemble the stiffness and mass matrices and solve the matrices equations, which makes the computation cost large, especially for the nonlinear PDEs as we have to assemble the stiffness matrix at every time step. Designing machine learning algorithms for the partially explicit scheme can alleviate this problem. Once the neural network is trained, the computation of solution becomes the multiplication of matrices and applying activation functions, which is more efficient in contrast to the finite element methods.

 The Equation (\ref{pe1}) is implicit for $u^{n+1}_{H,1}$, which makes the computations more complicated. We design a machine learning algorithm to predict $u^{n+1}_{H,1}$ and then we compute $u^{n+1}_{H,2}$ using Equation (\ref{pe2}). 


\begin{figure}[H]
\centering
\includegraphics[width = 10cm]{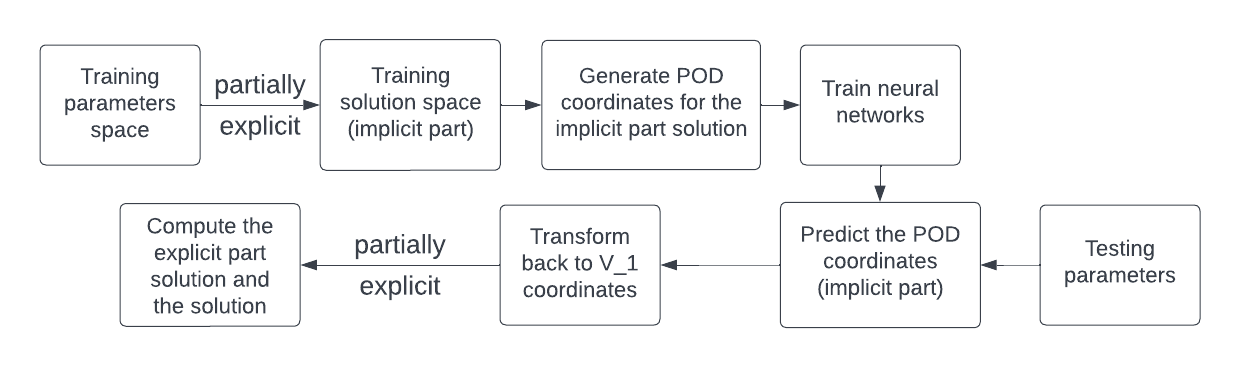}
\caption{Flowchart. }
\label{flowchart}
\end{figure}

We present a flowchart in Figure \ref{flowchart} to show the workflow of our method. We first need to generate the dataset which will be used to train the neural network. Let $w_m = (w_{m1}, w_{m2}, \cdots, w_{mk})$ be the parameters in $g_{w}(x,t)$. We set $w_{mi} \in [L,U]$ for $i = 1, 2,\cdots, k$ with $L$ and $U$ being the lower bound and upper bound, respectively. We generate the training space $\{w_m\}_{m=1}^M$ by drawing samples from a uniform distribution over the interval $[L,U]$. For every $w_m$ in the training space, we calculate its corresponding partially explicit solution using Equations (\ref{pe1}) and (\ref{pe2}). Let $N$ be the number of time steps and $M$ be the size of our training sample space. We denote this solution space as $\{\{u_{H,m}^n\}_{n=2}^N\}_{m=1}^M$, where $u_{H,m}^n = u_{H,m,1}^n + u_{H,m,2}^n$. We remark that the reason why $n$ starts from $2$  is that we need $u^0$ and $u^1$ to compute $u^2$ in the partially explicit scheme. The $u^1$ is computed by the Backward Euler scheme using basis functions from both $V_{H,1}$ and $V_{H,2}$. We only need $\{\{u_{H,m,1}^n\}_{n=2}^N\}_{m=1}^M$ for our machine learning algorithm. Before we move on, we discuss a useful tool called Proper Orthogonal Decomposition \cite{10.2307/24103957}.

For a given set of vectors $\{ x_i \}_{i=1}^c$ with $x_i \in \mathbb{R}^d$, we want to construct a smaller dimensional vector space that can represent the space spanned by $\{x_i\}_{i=1}^c$ accurately, i.e., we want to find orthonormal vectors $\{y_j\}_{j=1}^{l}$ that solve
\begin{align}
\min\limits_{y_1,\cdots,y_l } \sum_{j=1}^c \|x_j - \sum_{i=1}^l \langle x_j,y_l \rangle_{l^2} y_l \|^2_{l^2},
\label{minimization}
\end{align}
where $l\leq r$ and $r$ is the rank of $(x_1, x_2, \cdots, x_c)$.
We apply the singular value decomposition for this. Next, we briefly remind the POD procedure \cite{10.2307/24103957}. We first define a matrix $X = (x_1, x_2, \cdots, x_c)$ whose dimension is $d \times c$. Then $X^T X$ and $X X^T$ are $c\times c$ and $d\times d$ dimensional matrices. Consider the eigenvalues and eigenvectors of $X^T X$ as well as $X X^T$, 
\[ X^T X z_i = \sigma_i^2 z_i , \quad XX^T y_i = \sigma_i^2 y_i.   \]
Next, we define
\[ Y = (y_1, \cdots, y_d),\quad Z = (z_1, \cdots, z_c). \]
We now have 
\[ Y^T X Z = \Sigma, \]
where $\Sigma$ is the diagonal matrix whose main diagonal entries are $\sigma_i^2$. After doing rearrangement if necessary, we may assume 
\[ \sigma_1^2 \geq \sigma_2^2 \geq \cdots > 0. \]
The solution of (\ref{minimization}) is given by the $l$ eigenvectors of $X X^T$ corresponding to the largest $l$ eigenvalues.
Thus, the desired $l$ dimensional subspace  is the subspace spanned by eigenvectors corresponding to the largest $l$ eigenvalues.

There are $N\times M$ column vectors ($\in \mathbb{R}^{\dim(V_{H,1})}$) in the solution space $\{\{u_{H,m}^n\}_{n=2}^N\}_{m=1}^M$. We would like to output the solution $\{u_{H}^n\}_{n=2}^N$ at the same time so that we don't have to do the iteration for every time step recursively. In this case, the output dimension is $\dim(V_{H,1}) \times (N-1)$, which is too large ($\dim(V_{H,1})$ is $300$ in our examples). As a result, the number of parameters in the neural network will also be large, which increases the computation cost.  We define 
\begin{align*}
S = (u_{H,1,1}^2,u_{H,1,1}^3,\cdots, u_{H,1,1}^N, u_{H,2,1}^2, u_{H,2,1}^3, \cdots, u_{H,2,1}^N, \cdots, u_{H,M,1}^2, u_{H,M,1}^3, \cdots, u_{H,M,1}^N  ).
\end{align*}
Hence, $\dim(S) = \dim(V_{H,1}) \times ((N-1) \cdot M)$.
After the singular value decomposition to $S$, we define
\[ P = (p_1 , p_2 , \cdots, p_l), \]
where $p_i \in \mathbb{R}^{\dim(V_{H,1})}$ is the eigenvector corresponding to the $i$-th largest eigenvalue. We use $P$ to get a reduced dimensional coordinate vectors for every $u_{H,m,1}^n$,i.e.
\begin{align}
\label{PODdef}
u_{H,m,1}^{n,\prime} = (P^T P)^{-1} P^T u_{H,m,1}^n \in \mathbb{R}^l
\end{align} 
Thus, the output dimension is reduced to $l \times (N-1)$. In the testing stage, after we get the output $\{u_{H,1}^{n,\prime}\}_{n=2}^N$, we use $P$ to get the coarse grid solution $\{u_{H,1}^{n} = Pu_{H,1}^{n,\prime}\}_{n=2}^N$. 

As discussed above, in our neural networks, the inputs are the parameters $w$ in $g_{w}(x,t)$ and outputs are $\{u_{H,m,1}^{n,\prime}\}_n$. That is
\begin{align*}
\{u_{H,m,1}^{n,\prime}\}_n = \mathcal{N} (w),
\end{align*}
where $\mathcal{N}$ is the neural network. To better illustrate,
\begin{align*}
Y_1 = \phi_1(w W_1 + b_1),
\end{align*}
where $Y_1$, $\phi_1$, $w$, $W_1$ and $b_1$ are the output of the first layer, the activation function, inputs, weight matrix and the biased term. Here, $Y_1$, $w$ and $b_1$ are row vectors. Similarly,
\begin{align*}
Y_{i+1} = \phi_{i+1}( Y_i W_{i+1} + b_{i+1})
\end{align*}
and 
\begin{align*}
Y_I= Y_{I-1} W_I + b_I,
\end{align*}
where $I$ is the number of layers and $i = 1,2,\cdots, I-1$. Then, we can obtain $\{u_{H,m,1}^{n,\prime}\}_n$ by reshaping $Y_I$.

Next, we would like to prove two theoretical results. Under the setting in Section \ref{secML}, we first define the map from the input to the output of our neural network as
\begin{align*}
h : W &\rightarrow  \mathbb{R}^{l \times (N-1)},\\
 w &\mapsto ( u_{H,w,1}^{2,\prime,1}, u_{H,w,1}^{2,\prime,2}, \dots, u_{H,w,1}^{2,\prime,l}, u_{H,w,1}^{3,\prime,1}, u_{H,w,1}^{3,\prime,2}, \dots, u_{H,w,1}^{3,\prime,l}, \dots, u_{H,w,1}^{N,\prime,1}, u_{H,w,1}^{N,\prime,2}, \dots, u_{H,w,1}^{N,\prime,l} ),
\end{align*}
where $W$ is a compact subset of $\mathbb{R}^k$ with norm $\|\cdot\|_W$.
We will show that the neural network is capable of approximating $h$ precisely. We then show that when the loss function (mean squared error) goes to zero, the output of the neural network converges to the exact solution (which is computed via the partially explicit scheme). We state the following well-known theorem \cite{pmlr-v125-kidger20a}.  
\begin{theorem}
\textbf{Universal Approximation Theorem}. Let $\sigma : \mathbb{R} \rightarrow \mathbb{R}$ be any non-affine continuous function. Also, $\sigma$ is continuously differentiable at at least one point and with nonzero derivative at that point. Let $\mathcal{N}_{k,D}^{\sigma}$ be the space of neural networks with $k$ inputs, $D$ outputs, $\sigma$ as the activation function for every hidden layers and the identity function as the activation function for the output layers. Then given any $\epsilon > 0$ and any function $\mathcal{H} \in C\left(W, \mathbb{R}^{D}\right)$, there exists $\hat{\mathcal{H}} \in \mathcal{N}_{k,D}^{\sigma}$ such that 
\[ \sup_{w\in W} \| \hat{\mathcal{H}} (w) - \mathcal{H}(w) \|_{l_{\infty}(\mathbb{R{^{D}})}} < \epsilon.  \]
\end{theorem}

The universal approximation theorem relies on the continuity assumption of the mapping $\mathcal{H}$. More specifically, to apply the universal approximation theorem, we now need to show that $h$ is continuous.

\begin{lemma}
We define $\|u\|_a^2 = a(u,u) = (f(u),u)$ and assume $f$ is linear.
We know $V_H = V_{H,1} \oplus V_{H,2}$.
The two subspaces $V_{H,1}$ and $V_{H,2}$ are finite dimensional subspaces with trivial intersection. By the strengthened Cauchy Schwarz inequality \cite{aldaz2013strengthened}, there exists a constant $\gamma$ depending on $V_{H,1}$ and $V_{H,2}$ such that 
\begin{align}
0 < \gamma := \sup_{v_1 \in V_{H,1}, v_2 \in V_{H,2}} \cfrac{(v_1,v_2)} { \|v_1\| \|v_2\| } < 1. 
\label{s_CS}
\end{align}

If 
\begin{align}
\label{conditionDeltat}
\Delta t \leq (1-\gamma) \left( \sup_{v_2\in V_{H,2}} \cfrac{\|v_2\|_a^2} { \| v_2 \|^2 }  \right)^{-1},
\end{align}
 then $h$ is continuous with respect to $w\in W$.
 \label{thelemma}
\end{lemma}
The proof is given in Appendix \ref{appendixA}.
With the continuity of $h$, we have the following theorem.
\begin{theorem}
\label{theorem1}
For any $\epsilon > 0$, there exists $\hat{h} \in \mathcal{N}_{k,l(N-1)}^{\sigma}$ such that
    \[ \sup_{w\in W} \| \hat{h}(w) - h(w) \|_{l_{\infty}(\mathbb{R}^{l(N-1)})} < \epsilon/3. \] 
\end{theorem}
The theorem follows from the continuity of $h$ and the Universal Approximation Theorem. Finally, we present the estimation of the generalization error. 

\begin{theorem}
Let $W_T\subset W$ be the set containing all training samples. Fix $\epsilon > 0$. The functions $h$ and $\hat{h}$ are uniformly continuous since they are continuous and $W$ is compact.
It follows that there exists $ \delta > 0$ such that
\begin{align*}
\| \hat{h}(w^{\prime}) - \hat{h} (w^{\prime\prime}) \|_{l_{\infty}} < \epsilon/3, \quad
\| h(w^{\prime}) - h (w^{\prime\prime}) \|_{l_{\infty}} < \epsilon/3,
\end{align*}
for any $w', w''\in W$ and $\|w'-w''\|<\delta$.
We assume that for any $w\in W$, there exists $w_i \in W_T$ such that $\| w - w_i \|_W < \delta$. Let $\hat{h} \in \mathcal{N}_{k,l(N-1)}^{\sigma}$ be the mapping in Theorem \ref{theorem1}. Then we have
\[ \sup_{w\in W} \| \hat{h}(w) - h(w) \|_{l_{\infty}(\mathbb{R}^{l(N-1)})} < \epsilon. \]
\end{theorem}

\begin{proof}
For any $w\in W$, we can find $w_i \in W_T$ such that $\| w - w_i \|_W < \delta$, it follows that:
\begin{align*}
\| \hat{h} (w) - h(w) \|_{l_{\infty}} 
&= \| \hat{h} ( w ) - \hat{h}(w_i) + \hat{h}(w_i) - h(w_i) + h(w_i) - h(w) \|_{l_{\infty}} \\
&\leq  \| \hat{h} ( w ) - \hat{h}(w_i) \|_{l_{\infty}} + \| \hat{h}(w_i) - h(w_i) \|_{l_{\infty}} 
		+ \| h(w_i) - h(w) \|_{l_{\infty}} \\
&\leq \epsilon/3 + \epsilon/3 + \epsilon/3 = \epsilon.
\end{align*}

As $\epsilon$ and $w$ are arbitrary, we obtain the conclusion. 
\end{proof}



\section{Numerical Results}
\label{secNR}
In this section, we present numerical examples.
In all cases, the domain $\Omega$ is $[0,1]^2$. The permeability field $\kappa$ we use is shown in Figure \ref{figkappa}. Note that there are high contrast channels in the permeability field. The fine mesh size is $\sqrt{2}/100$ and the coarse mesh size is $\sqrt{2}/10$.

\begin{figure}[H]
\centering
\includegraphics[width = 6cm]{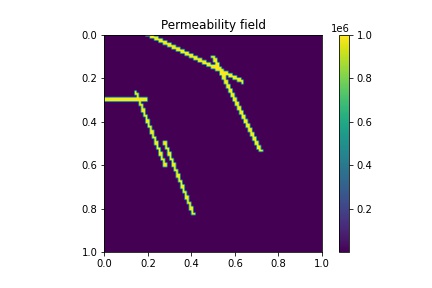}
\caption{The permeability field $\kappa$.}
\label{figkappa}
\end{figure}

We define the following notation.

\begin{itemize}
    \item $u_{H,l}$ is obtained by predicting $u_{H,1}$ and computing $u_{H,2}$ by Equation (\ref{pe2}).
    \item $u_{H,c}$ is obtained by computing $u_{H,1}$ and $u_{H,2}$ by Equation (\ref{pe1}) and (\ref{pe2}).
    \item $u_{H,f}$ is obtained using the fine grid basis functions and Equation (\ref{refsol1}). 
\end{itemize}
We also need to define the following relative error for better illustration.
Let $\|\cdot\|$ be the $L^2$ norm. We know that $u_{H,l}=u_{H,l,1}+u_{H,l,2}$ and $u_{H,c} = u_{H,c,1} + u_{H,c,2}$.
\begin{itemize}
\item $e_1 = \cfrac{\|u_{H,l}-u_{H,f}\|}{\|u_{H,f}\|} \times 100\%$,\quad
      $e_2 = \cfrac{\|u_{H,c}-u_{H,f}\|}{\|u_{H,f}\|} \times 100\%$.
\item $e_3 = \cfrac{\|u_{H,l,1}-u_{H,c,1}\|}{\|u_{H,c,1}\|} \times 100\%$,\quad
      $e_4 = \cfrac{\|u_{H,l}-u_{H,c}\|}{\|u_{H,c}\|} \times 100\%$.
\end{itemize}
$e_1$ and $e_2$ are used to measure the difference of accuracy between our learning scheme and the scheme in which we compute both $u_{H,1}$ and $u_{H,2}$. $e_3$ and $e_4$ are defined to measure how close the learning output is to their target.
The neural network is fully connected and is activated by the scaled exponential linear unit (SELU). To obtain the solution of the minimization problem, we use the gradient-based optimizer Adam \cite{kingma2017adam}.
The loss function we use is the mean square error (MSE). Our machine learning algorithm is implemented using the open-source library Keras \cite{chollet2015keras} with the Tensorflow \cite{tensorflow2015-whitepaper} backend.
The neural network we use is a simple neural network with $5$ fully connected layers. We remark that  with this  neural network, we can obtain high accuracy.

In the first two examples, we consider the following linear parabolic 
partial differential equations
\begin{align*}
    u_t - \nabla \cdot (\kappa \nabla u ) &= g_{w}(x,t), \;(x,t) \in \Omega\times [0,T] , \\
    \nabla u(x) \cdot \vec{n} &= 0, \; x \in \partial \Omega, \\
    u(x,0) &= u_0,
\end{align*}
where $\vec{n}$ is the outward normal vector and $u_0$ is the initial condition which is presented it in Figure \ref{figIC&ST1}. 
In our first example, the source term is defined as follows,
\begin{align*}
    g_{w}(x,t) = 
    \begin{cases}
    100(w_1 \sin( \cfrac{2\pi t}{T}) + w_2 \sin(\cfrac{5.2\pi t}{T}))  & \text { for } x \in[0.2,0.3]^2 \\ 
    100(w_3 \sin( \cfrac{2.4\pi t}{T}) + w_4 \sin(\cfrac{4\pi t}{T}))  & \text { for } x \in[0.8,0.9]^2 \\ 
    0 & \text { otherwise, }
    \end{cases}
\end{align*}
where $(w_1, w_2, w_3, w_4)$ are parameters. In this case, we let $1\leq w_i \leq 10$ for $i=1,2,3,4$. We show a simple sketch of $g_{w}(x,t)$ in Figure \ref{figIC&ST1} to illustrate. This sketch only shows the position of the coarse blocks where the source term is nonzero. We choose this source, which has some similarities to the wells in petroleum engineering. The final time is $T = 0.01$, there are $N=100$ time steps and the time step size is $\Delta t = T/N = 10^{-4}$.

\begin{figure}[H]
\centering
\includegraphics[width = 6cm]{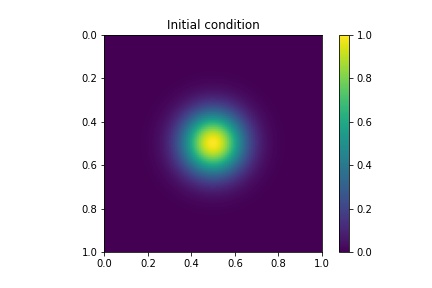}
\includegraphics[width = 6cm]{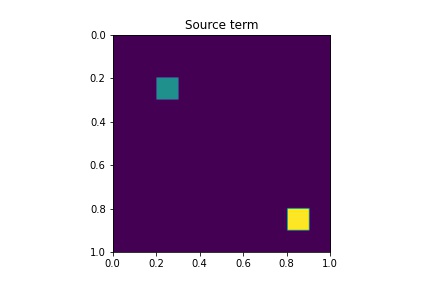}
\caption{Left: Initial condition $u_0$. Right: Source term $g(x,t)$. The square boxes indicate the positions of the wells.}
\label{figIC&ST1}
\end{figure}

The training space contains $1000$ samples and the testing space contains $500$ samples. Every sample's parameter $w$ is drawn from the uniform distribution over $[1,10]$. The number of POD basis vectors is $l=15$. We choose this $l$ based on the eigenvalue behavior in the process of singular value decomposition. We present $e_1$ and $e_2$ defined before in Figure \ref{error1}. The error is calculated at each time step as the average over all samples in the testing space. From the plot in Figure \ref{error1}, we find that the curves for $e_1$ and $e_2$ coincide, which means our learning algorithms can obtain similar accuracy as the scheme in which we compute $u_H$ by Equation (\ref{pe1}) and (\ref{pe2}).
The average of $e_3$ and $e_4$ over time are $0.151\%$ and $0.155\%$, respectively, which means the output of our learning scheme nearly resembles the corresponding target (the solution we obtain by computing both equations).  
\begin{figure}[H]
\centering
\includegraphics[width = 6cm]{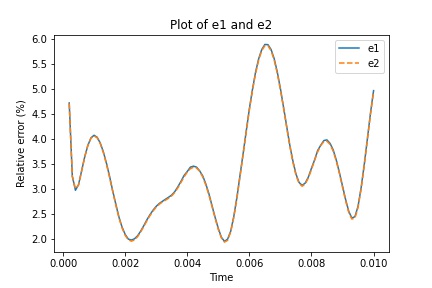}
\caption{$e_1$ and $e_2$ of the linear case with $4$ parameters.}
\label{error1}
\end{figure}

In the next case, we test our algorithms using a more complicated source term with more parameters. The source term we use is,
\begin{align*}
    g_{w}(x,t) = 
    \begin{cases}
    100(w_1 \sin( \cfrac{\pi t}{T}) + w_2 \cos(\cfrac{3.2\pi t}{T}) )  & \text { for } x \in[0.2,0.3]^2 \\ 
    100(w_3 \sin( \cfrac{2.2\pi t}{T}) + w_4 \sin(\cfrac{1.6\pi t}{T}) ) & 
    \text { for } x \in[0.8,0.9]^2 \\ 
    100(w_5 \cos( \cfrac{3\pi t}{T}) + w_6 \cos(\cfrac{4.6\pi t}{T}) ) & \text { for } x \in[0.2,0.3] \times [0.8,0.9] \\ 
    100(w_7 \cos( \cfrac{1.4\pi t}{T}) + w_8 \sin(\cfrac{5\pi t}{T}) ) & \text { for } x \in[0.8,0.9] \times [0.2,0.3] \\ 
    100(w_9 \sin( \cfrac{2.8\pi t}{T}) + w_{10} \sin(\cfrac{4\pi t}{T}) ) & \text { for } x \in[0.5,0.6]^2 \\ 
    0 & \text { otherwise, }
    \end{cases}
\end{align*}
where $(w_1, w_2, w_3,\cdots,w_{10})$ are parameters. In this case, the parameters $w_i$ of every sample are drawn from the uniform distribution over $[1,20]$. We present an easy sketch of $g_{w}(x,t)$ in Figure \ref{figIC&ST2} to illustrate the positions of the sources. The initial condition $u_0$ is also shown in Figure \ref{figIC&ST2}. The final time $T$, the number of time steps $N$ and the time step size are $0.01$, $100$ and $10^{-4}$, respectively. 

\begin{figure}[H]
\centering
\includegraphics[width = 6cm]{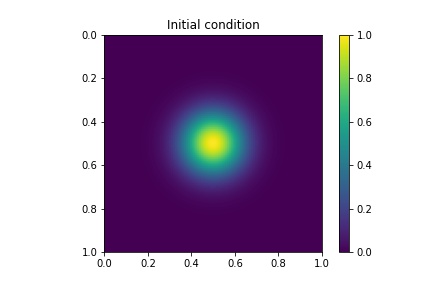}
\includegraphics[width = 6cm]{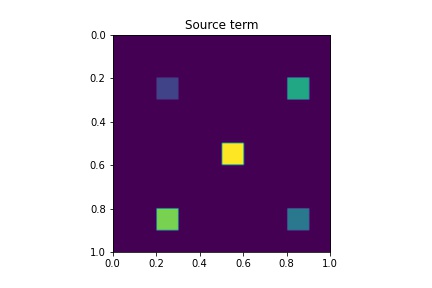}
\caption{Left: Initial condition $u_0$. Right: Source term $g_{w}(x,t)$. The square boxes in the plot indicate the positions of the wells.}
\label{figIC&ST2}
\end{figure}

The model is trained using 2000 samples and we test it with 1000 samples.
The number of POD basis vectors is $l = 25$. In Figure \ref{error2}, the $e_1$ and $e_2$ are presented. 
By observing the error plot, we find that $e_1$ and $e_2$ are close to each other, which implies that our learning algorithms can obtain similar accuracy as the scheme where we calculate both $u_{H,1}$ and $u_{H,2}$ using Equation (\ref{pe1}) and (\ref{pe2}). The average of $e_3$ and $e_4$ over time are $0.673\%$ and $0.677\%$, respectively. Such small $e_3$ and $e_4$ indicate that the prediction is accurate.
\begin{figure}[H]
\centering
\includegraphics[width = 6cm]{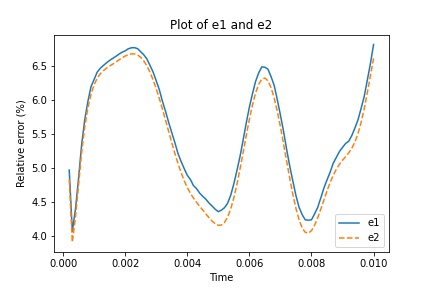}
\caption{$e_1$ and $e_2$ of the linear case with $10$ parameters.}
\label{error2}
\end{figure}

In the next case, we consider the following equation
\begin{align*}
u_t - \nabla \cdot ( \kappa e^u \nabla u ) &= g(x,t), \; (x,t) \in \Omega \times [0,T],\\
\nabla u(x) \cdot \vec{n} = 0 , \; x \in \partial \Omega ,\\
u(x,0) = u_0,
\end{align*}
where $\vec{n}$ is the outward normal vector. This equation is nonlinear and its computation cost is much larger than the linear case because we have to assemble the stiffness matrix at every time step. We use the same initial condition $u_0$ and source term $g(x,t)$ as the second example. The final time is $T = 0.001$, the number of time steps is $N = 40$ and the time step size is $\Delta t = T/N =  2.5\times 10^{-5}$. We train our neural network using $1000$ samples and the network is tested using $500$ samples. In Figure \ref{error3}, we present $e_1$ and $e_2$ defined before. From the first plot, we find that the curves for $e_1$ and $e_2$ nearly coincide, which implies that our machine learning algorithms can have the same accuracy as we compute the solution by  Equations (\ref{pe1}) and (\ref{pe2}). The average of $e_3$ and $e_4$ over time are $0.460\%$ and $0.450\%$, respectively, which means our neural network can approximate the target well and achieve a fantastic accuracy.

\begin{figure}[H]
\centering
\includegraphics[width = 6cm]{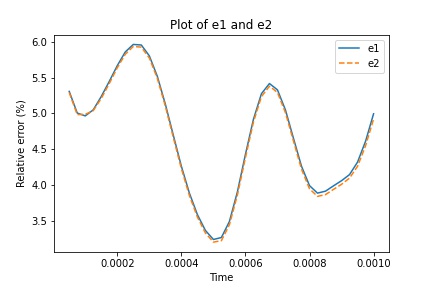}
\caption{$e_1$ and $e_2$ of the nonlinear case with 10 parameters.}
\label{error3}
\end{figure}

\section{Conclusion}

In this paper, we construct neural networks based on the partially explicit splitting scheme. We predict $u_1$, the implicit part of the solution, while solving for $u_2$, the explicit part of the solution. The output dimension resulting
from $u_1$ is  large and we apply the POD  to reduce the dimension. The POD transformation greatly reduces the output dimension and at the same time the number of parameters in the neural networks. The POD error is small compared to the overall error. We discuss three numerical examples. We present a theoretical justification. We conclude that our machine learning algorithms can obtain similar accuracy as the partially explicit scheme.

\appendix
\section{Proof of Lemma \ref{thelemma}}
\label{appendixA}
In this section, we prove the Lemma \ref{thelemma}, i.e., the function 
\begin{align*}
h : W &\rightarrow  \mathbb{R}^{l \times (N-1)},\\
 w &\mapsto ( u_{H,w,1}^{2,\prime,1}, u_{H,w,1}^{2,\prime,2}, \dots, u_{H,w,1}^{2,\prime,l}, u_{H,w,1}^{3,\prime,1}, u_{H,w,1}^{3,\prime,2}, \dots, u_{H,w,1}^{3,\prime,l}, \dots, u_{H,w,1}^{N,\prime,1}, u_{H,w,1}^{N,\prime,2}, \dots, u_{H,w,1}^{N,\prime,l} )
 \end{align*}
is continuous.
We equip $\mathbb{R}^{l \times (N-1)}$ with the $l_{\infty}$ norm. 
Let us decompose the mapping as $h = \mathcal{P} \circ h_2 \circ h_1$, where
\begin{align*}
&h_1 : (W, \|\cdot\|_{\infty}) \to (L^2(\Omega \times [0,T],\mathbb{R} ), \|\cdot\|_{L^2}) \quad w \mapsto g_w(x,t), \\
&h_2 : (L^2(\Omega \times [0,T],\mathbb{R} ), \|\cdot\|_{L^2})  \to (H^1(\Omega), \|\cdot\|_{H^1} ), \quad g_w(x,t) \mapsto \{u_H(x,t^n)\}_{n=1}^N,
\end{align*}
and $\mathcal{P}$ is the POD transformation defined by (\ref{PODdef}).
The map $\mathcal{P}$ is a projection, thus continuous. We assume that $g_{w}$ is continuous with respect to $w$ and therefore, $h_1$ is continuous. To show that $h$ is continuous, we only need to prove $h_2$ is continuous. The proof is given in the next lemma.

\begin{lemma}
Let us assume $f$ is linear, and we define a bilinear form $a(u,v) := (f(u),v)$ and its co-responding norm as $\|u\|_a := \sqrt{a(u,u) }$.
If 
\begin{align}
\label{conditionDeltatAppendix}
\Delta t \leq (1-\gamma) \left( \sup_{v_2\in V_{H,2}} \cfrac{\|v_2\|_a^2} { \| v_2 \|^2 }  \right)^{-1},
\end{align}
the function $h_2$ is continuous with respect to the source term $g_w(x,t)$.
\end{lemma}
\begin{proof} 
\textbf{Case I.} The solution $u_i(x,\Delta t)$ for $i = 1, 2$ is continuous with respect to the source term $g_w(x,t)$.\\
The initial condition $u(x,0)$ is given and the solution at time step $1\cdot \Delta t$ is calculated via the fully implicit scheme. Thus, we want to show the continuity from the source term to the solution $u_i^1$ of the following scheme
\begin{align*}
    (\frac{u^{1}_{H}- u^0_H}{\Delta t } , v) + ( f(u^{1}_H) , v) = 
    (g_{w}(x,t^{1}) , v) \quad \forall v \in V_H,
\end{align*}
where $V_H = V_{H,1} \oplus V_{H,2}$.

In the following proof, we abbreviate the subscript $H$ for simplicity.

Let $w_1$ and $w_2$ be different, and let us denote the corresponding sources as $g_{w_1}$ and $g_{w_2}$. The discretized solutions $u^{1}_{w_1}$ and $u^{1}_{w_2}$ evolved for one time step then satisfy:
\begin{align*}
    (\frac{u^{1}_{w_1}- u^0}{\Delta t } , v) + ( f(u^{1}_{w_1}) , v) = 
    (g_{w_1}(x,t^{1}) , v) \quad \forall v \in V_H,\\
  (\frac{u^{1}_{w_2}- u^0}{\Delta t } , v) + ( f(u^{1}_{w_2}) , v) = 
    (g_{w_2}(x,t^{1}) , v) \quad \forall v \in V_H.
\end{align*}
It follows immediately that, 
\begin{align*}
(\frac{u^{1}_{w_1} - u^{1}_{w_2} }{\Delta t} ,v ) + (f(u^{1}_{w_1}) -  f(u^{1}_{w_2}), v) = (g_{w_1}(x,t^{1}) - g_{w_2}(x,t^{1}), v) \quad \forall v \in V_H.
\end{align*}
Test with $v = u^{1}_{w_1} - u^{1}_{w_2}$, the linearity of $f$ then yields,
\begin{align*}
\begin{split}
\frac{1}{\Delta t} \|u^{1}_{w_1} - u^{1}_{w_2}\|^2 + (f(u^{1}_{w_1} - u^{1}_{w_2}), u^{1}_{w_1} - u^{1}_{w_2} ) &\leq \| g_{w_1}(x,t^{1}) - g_{w_2}(x,t^{1}) \| \|u^{1}_{w_1} - u^{1}_{w_2}\| \\
& \leq \frac{\Delta t}{2} \| g_{w_1}(x,t^{1}) - g_{w_2}(x,t^{1}) \|^2 + 
		\frac{1}{2\Delta t} \|u^{1}_{w_1} - u^{1}_{w_2}\|^2
\end{split}
\end{align*}

After rearrangement, we have 
\begin{align*}
\begin{split}
\frac{1}{2 \Delta t} \|u^{1}_{w_1} - u^{1}_{w_2}\|^2 + \|u^{1}_{w_1} - u^{1}_{w_2}\|_a^2 
 \leq \frac{\Delta t}{2} \| g_{w_1}(x,t^{1}) - g_{w_2}(x,t^{1}) \|^2.
\end{split}
\end{align*}

Apply the strengthened Cauchy Schwartz inequality (\ref{s_CS}), let $u_{w_m}^{n} = u_{1,w_m}^{n} + u_{2, w_m}^{n}$ with $u_{1,w_m}^{n} \in V_{H,1}$ and $u_{2,w_m}^{n} \in V_{H,2}$, we have,
\begin{align*}
\|  u_{w_1}^{n} - u_{w_2}^{n}  \|^2 &= \| u_{1,w_1}^{n} - u_{1, w_2}^{n}  \|^2 + \| u_{2,w_1}^{n} - u_{2,w_2}^{n}  \|^2 + 2 ( u_{1,w_1}^{n} - u_{1, w_2}^{n} ,  u_{2,w_1}^{n} - u_{2,w_2}^{n}  ) \\
&\geq \| u_{1,w_1}^{n} - u_{1, w_2}^{n}  \|^2 + \| u_{2,w_1}^{n} - u_{2,w_2}^{n}  \|^2 - 2\gamma ( u_{1,w_1}^{n} - u_{1, w_2}^{n} ,  u_{2,w_1}^{n} - u_{2,w_2}^{n}  ) \\
&\geq (1- \gamma)  (\| u_{1,w_1}^{n} - u_{1, w_2}^{n}  \|^2 + \| u_{2,w_1}^{n} - u_{2,w_2}^{n}  \|^2 ),
\end{align*}
for all $n>0$.

Thus,
\begin{align}
\frac{1-\gamma}{2 \Delta t} \sum_{i} \|u^{1}_{i,w_1} - u^{1}_{i,w_2}\|^2 + \|u^{1}_{w_1} - u^{1}_{w_2}\|_a^2 
 \leq \frac{\Delta t}{2} \| g_{w_1}(x,t^{1}) - g_{w_2}(x,t^{1}) \|^2.
 \label{stability1}
\end{align}

\textbf{Case II.}
The solution $u_1(x, n\Delta t)$ is continuous with respect to the source term $g_w(x,t)$, for $n = 2, 3, \dots, N$.

In the partially explicit scheme, for the source term $g_{w_1}$, the solution $u_{i,w_1}^{n+1}$ satisfies
\begin{align}
&(\frac{u_{1,w_1}^{n+1}-u_{ 1,w_1}^{n}}{\Delta t}+\frac{u_{ 2,w_1}^{n}-u_{ 2,w_1}^{n-1}}{\Delta t}, v_{1})+(f(u_{ 1,w_1}^{n+1}+u_{ 2,w_1}^{n}),v_1) = (g_{w_1}(x,t^n), v_{1}) \quad\forall v_{1} \in V_{H, 1}, \label{pfpe1} \\
&(\frac{u_{2,w_1}^{n+1}-u_{2,w_1}^{n}}{\Delta t} + \frac{u_{ 1,w_1}^{n}-u_{ 1,w_1}^{n-1}}{\Delta t}, v_{2})+(f(u_{1,w_1}^{n+1}+u_{2, w_1}^{n}),v_2) = (g_{w_1}(x,t^n), v_{2}) \quad\forall v_{2} \in V_{H, 2} .\label{pfpe2} 
\end{align}

Similarly, for $g_{w_2}$, the solution $u_{i,w_2}^{n+1}$ satisfies 
\begin{align}
&(\frac{u_{1,w_2}^{n+1}-u_{ 1,w_2}^{n}}{\Delta t}+\frac{u_{ 2,w_2}^{n}-u_{ 2,w_2}^{n-1}}{\Delta t}, v_{1})+(f(u_{ 1,w_2}^{n+1}+u_{ 2,w_2}^{n}),v_1) = (g_{w_2}(x,t^n), v_{1}) \quad\forall v_{1} \in V_{H, 1}, \label{pfpe3} \\
&(\frac{u_{2,w_2}^{n+1}-u_{2,w_2}^{n}}{\Delta t} + \frac{u_{ 1,w_2}^{n}-u_{ 1,w_2}^{n-1}}{\Delta t}, v_{2})+(f(u_{1,w_2}^{n+1}+u_{2, w_2}^{n}),v_2) = (g_{w_2}(x,t^n), v_{2}) \quad\forall v_{2} \in V_{H, 2} .\label{pfpe4} 
\end{align}
Let us denote $r^{n}_i = u_{i,w_1}^{n} - u_{i,w_2}^{n} $ for $i = 1, 2$, and subtract (\ref{pfpe3}) from (\ref{pfpe1}) and (\ref{pfpe4}) from (\ref{pfpe2}), it follows that,
\begin{align*}
(\frac{r_1^{n+1}  - r_1^n }{\Delta t}
+\frac{r_2^n - r_{2}^{n-1} }{\Delta t}, v_{1})
+(f(u_{ 1,w_1}^{n+1}+u_{ 2,w_1}^{n}) - f(u_{ 1,w_2}^{n+1}+u_{ 2,w_2}^{n}) ,v_1) \\
= (g_{w_1}(x,t^n) - g_{w_2}(x,t^n), v_{1}) \quad\forall v_{1} \in V_{H, 1},  \\
(\frac{r_{2}^{n+1} - r_{2}^{n} }{\Delta t}
+\frac{r_{1}^{n} - r_{1}^{n-1} }{\Delta t}, v_{2})
+(f(u_{ 1,w_1}^{n+1}+u_{ 2,w_1}^{n}) - f(u_{ 1,w_2}^{n+1}+u_{ 2,w_2}^{n}) ,v_2) \\
= (g_{w_1}(x,t^n) - g_{w_2}(x,t^n), v_{2}) \quad\forall v_{2} \in V_{H, 2}.
\end{align*}

The linearity of $f$ then yields,
\begin{align*}
(\frac{r_{1}^{n+1} - r_{1}^{n} }{\Delta t}
+\frac{r_2^n - r_{2}^{n-1} }{\Delta t}, v_{1})
+(f(r_{1}^{n+1} + r_{2}^{n} ) ,v_1) 
= (g_{w_1}(x,t^n) - g_{w_2}(x,t^n), v_{1}) \quad\forall v_{1} \in V_{H, 1},  \\
(\frac{r_{2}^{n+1} - r_{2}^{n} }{\Delta t}
+\frac{r_{1}^{n} - r_{1}^{n-1} }{\Delta t}, v_{2})
+(f(r_{1}^{n+1} + r_{2}^{n} ) ,v_2) 
= (g_{w_1}(x,t^n) - g_{w_2}(x,t^n), v_{2}) \quad\forall v_{2} \in V_{H, 2}.
\end{align*}

Let $v_1 = r_{1}^{n+1} - r_1^n$ and 
$v_2 = r_{2}^{n+1} - r_2^n$. We get
\begin{align*}
(\frac{r_{1}^{n+1} - r_{1}^{n} }{\Delta t}
+\frac{r_2^n - r_{2}^{n-1} }{\Delta t}, 
r_{1}^{n+1} - r_{1}^{n} )
+(f(r_{1}^{n+1} + r_{2}^{n} ) ,
r_{1}^{n+1} - r_{1}^{n} ) 
= (g_{w_1}(x,t^n) - g_{w_2}(x,t^n), r_{1}^{n+1} - r_{1}^{n} ), \\ 
(\frac{r_{2}^{n+1} - r_{2}^{n} }{\Delta t}
+\frac{r_{1}^{n} - r_{1}^{n-1} }{\Delta t}, 
 r_{2}^{n+1} - r_2^n )
+(f(r_{1}^{n+1} + r_{2}^{n} ) ,
 r_{2}^{n+1} - r_2^n) 
= (g_{w_1}(x,t^n) - g_{w_2}(x,t^n),  r_{2}^{n+1} - r_2^n ).
\end{align*}

Sum up these two equations, we have
\begin{align}
\frac{1}{\Delta t} \sum_i \| r_{i}^{n+1} -  r_{i}^{n} \|^2 
 + \frac{1}{\Delta t} \sum_{i \neq j}  ( r_i^{n+1} - r_i^n, r_j^{n} - r_j^{n-1} )  
 + (f(r_{1}^{n+1} + r_{2}^{n} ) ,
 r^{n+1} - r^n ) \nonumber \\ 
= (g_{w_1}(x,t^n) - g_{w_2}(x,t^n),  r^{n+1} - r^n ),
\label{eqn1}
\end{align}
where $r^n = r^n_1 + r^n_2$.
The second term of (\ref{eqn1}) on the left hand side can be further estimated as,
\begin{align*}
\frac{1}{\Delta t} \sum_{i \neq j}  |( r_{i}^{n+1} - r_{i}^n ,
	r_{j}^{n} - r_{j}^{n-1}  )  | 
&\leq \frac{\gamma}{\Delta t} \sum_{i \neq j} 
\| r_{i}^{n+1} - r_{i}^n \| 
\| r_{j}^{n} - r_{j}^{n-1} \| \\
&\leq \frac{\gamma}{2 \Delta t} \sum_{i \neq j} 
( \| r_{i}^{n+1} - r_{i}^n \|^2  
+ \| r_{j}^{n} - r_{j}^{n-1} \|^2 ),
\end{align*}
where we apply the strenthened Cauchy Schwartz inequality. It follows that,
\begin{align*}
&\frac{1}{\Delta t} \sum_i \| r_{i}^{n+1} -  r_{i}^{n} \|^2 
 + \frac{1}{\Delta t} \sum_{i \neq j}  ( r_{i}^{n+1} - r_{i}^n ,
	r_{j}^{n} - r_{j}^{n-1}  )   \\
\geq& 
\frac{2-\gamma}{2\Delta t} \sum_i \| r_{i}^{n+1} -  r_{i}^{n} \|^2 
- \frac{\gamma} { 2\Delta t} \sum_i \| r_{i}^{n} - r_{i}^{n-1} \|^2 \\
=& (\frac{\gamma}{2\Delta t} + \frac{1-\gamma}{\Delta t}) 
\sum_i \| r_{i}^{n+1} -  r_{i}^{n} \|^2 
- \frac{\gamma} { 2\Delta t} \sum_i \| r_{i}^{n} - r_{i}^{n-1} \|^2.
\end{align*}
Note that $( f(r_{1}^{n+1} + r_{2}^{n} ) , r^{n+1} - r^n ) =
 a( r_{1}^{n+1} + r_{2}^{n}  , r^{n+1} - r^n )$, we then can estimate the third term of (\ref{eqn1}),
\begin{align*}
&(  f(r_{1}^{n+1} + r_{2}^{n} ) , r^{n+1} - r^n ) \\ 
= &a( r^{n+1}   ,
 r^{n+1} - r^n )
 - a ( r_2^{n+1} - r_{2}^{n}  ,
 r^{n+1} - r^n )  \\
=& \frac{1}{2} ( \| r^{n+1} \|_a^2 
+ \| r^{n+1} - r^n\|^2_a  
- \| r^n \|_a^2 ) 
 - a ( r_2^{n+1} - r_{2}^{n}  ,
 r^{n+1} - r^n )  \\
\geq& \frac{1}{2} ( \| r^{n+1} \|_a^2 
+ \| r^{n+1} - r^n\|^2_a  
- \| r^n \|_a^2 ) 
 - \frac{1}{2} \| r_2^{n+1} - r_{2}^{n} \|_a^2 
 -\frac{1}{2}  \| r^{n+1} - r^n \|_a^2   \\
 =& \frac{1}{2} ( \| r^{n+1} \|_a^2  
- \| r^n \|_a^2 ) 
 - \frac{1}{2} \| r_2^{n+1} - r_{2}^{n} \|_a^2. 
\end{align*}
Now (\ref{eqn1}) can be estimated as:
\begin{align*}
(\frac{\gamma}{2\Delta t} + \frac{1-\gamma}{\Delta t})
 \sum_i \| r_{i}^{n+1} -  r_{i}^{n} \|^2 
- \frac{\gamma} { 2\Delta t} \sum_i \| r_{i}^{n} - r_{i}^{n-1} \|^2 
+ \frac{1}{2} ( \| r^{n+1} \|_a^2  
- \| r^n \|_a^2 ) 
 - \frac{1}{2} \| r_2^{n+1} - r_{2}^{n} \|_a^2 \\
\leq 
(g_{w_1}(x,t^n) - g_{w_2}(x,t^n),  r^{n+1} - r^n ).
\end{align*}

After rearranging, we have
\begin{align*}
&\frac{\gamma}{2\Delta t}
 \sum_i \| r_{i}^{n+1} -  r_{i}^{n} \|^2 
+ \frac{1}{2}  \| r^{n+1} \|_a^2  +
 \frac{1-\gamma}{\Delta t}
 \sum_i \| r_{i}^{n+1} -  r_{i}^{n} \|^2
 - \frac{1}{2} \| r_2^{n+1} - r_{2}^{n} \|_a^2 \\
\leq &
\frac{\gamma} { 2\Delta t} \sum_i \| r_{i}^{n} - r_{i}^{n-1} \|^2
+ \frac{1}{2}  \| r^n \|_a^2  
+ (g_{w_1}(x,t^n) - g_{w_2}(x,t^n),  r^{n+1} - r^n ) \\
\leq &
\frac{\gamma} { 2\Delta t} \sum_i \| r_{i}^{n} - r_{i}^{n-1} \|^2
+ \frac{1}{2}  \| r^n \|_a^2  
+ \gamma \| g_{w_1}(x,t^n) - g_{w_2}(x,t^n) \| \|  r^{n+1} - r^n \|  \\
\leq &
\frac{\gamma} { 2\Delta t} \sum_i \| r_{i}^{n} - r_{i}^{n-1} \|^2
+ \frac{1}{2}  \| r^n \|_a^2  
+ \frac{\gamma^2  \Delta t } {1 - \gamma} \| g_{w_1}(x,t^n) - g_{w_2}(x,t^n) \|^2 
+ \frac{1-\gamma}{4 \Delta t} \|  r^{n+1} - r^n \|^2 \\
\leq &
\frac{\gamma} { 2\Delta t} \sum_i \| r_{i}^{n} - r_{i}^{n-1} \|^2
+ \frac{1}{2}  \| r^n \|_a^2  
+ \frac{\gamma^2  \Delta t } {1 - \gamma} \| g_{w_1}(x,t^n) - g_{w_2}(x,t^n) \|^2 
+  \frac{1-\gamma}{2 \Delta t} \sum_i \|  r_{i}^{n+1} - r_i^n \|^2.
\end{align*}
We assume the CFL condition (\ref{conditionDeltatAppendix}), specifically,
\begin{align*}
 (\frac{1-\gamma}{\Delta t} - \frac{1-\gamma}{2 \Delta t} ) 
 \sum_i \| r_{i}^{n+1} -  r_{i}^{n} \|^2
 - \frac{1}{2} \| r_2^{n+1} - r_{2}^{n} \|_a^2 \geq 0.
\end{align*}
This yields,
\begin{align}
\begin{split}
&\frac{\gamma}{2\Delta t}
 \sum_i \| r_{i}^{n+1} -  r_{i}^{n} \|^2 
+ \frac{1}{2}  \| r^{n+1} \|_a^2  \\
\leq &
\frac{\gamma} { 2\Delta t} \sum_i \| r_{i}^{n} - r_{i}^{n-1} \|^2
+ \frac{1}{2}  \| r^n \|_a^2  
+ \frac{\gamma^2  \Delta t } {1 - \gamma}  \| g_{w_1}(x,t^n) - g_{w_2}(x,t^n) \|^2  \\
\leq &
\frac{\gamma} { 2\Delta t} \sum_i \| u_{ i,w_1}^{1}  - u_{ i,w_2}^{1} - (u_{ i,w_1}^{0} - u_{ i,w_2}^{0}) \|^2
+ \frac{1}{2}  \| u_{w_1}^{1} - u_{w_2}^{1} \|_a^2  
+ \frac{\gamma^2  \Delta t } {1 - \gamma} \sum_{k=1}^n \| g_{w_1}(x,t^k) - g_{w_2}(x,t^k) \|^2\\
= &
\frac{\gamma} { 2\Delta t} \sum_i \| u_{ i,w_1}^{1}  - u_{ i,w_2}^{1}\|^2 
+ \frac{1}{2}  \| u_{w_1}^{1} - u_{w_2}^{1} \|_a^2  
+ \frac{\gamma^2  \Delta t } {1 - \gamma} \sum_{k=1}^n \| g_{w_1}(x,t^k) - g_{w_2}(x,t^k) \|^2 \\
\leq & \frac{\gamma \Delta t}{2(1-\gamma)}  \| g_{w_1}(x,t^1) - g_{w_2}(x, t^1) \|^2
+  \frac{\gamma^2  \Delta t } {1 - \gamma} \sum_{k=1}^n \| g_{w_1}(x,t^k) - g_{w_2}(x,t^k) \|^2,
\label{ineq-1}
\end{split}
\end{align}
where we use (\ref{stability1}) in the last inequality.
Then, 
\begin{align*}
& \frac{\gamma}{2\Delta t} \sum_i \| r_i^{n+1} \|^2 
\leq  \frac{\gamma}{ \Delta t} \sum_i (\| r_i^{n+1} - r_i^n \|^2 +  \| r_i^n \|^2)  \\
\leq & \frac{\gamma \Delta t }{1-\gamma} \| g_{w_1}(x,t^1) - g_{w_2}(x, t^1) \|^2
+ \frac{2 \gamma^2  \Delta t } {1 - \gamma} \sum_{k=1}^n \| g_{w_1}(x,t^k) - g_{w_2}(x,t^k) \|^2 + \frac{\gamma}{\Delta t} \sum_i \| r_i^{n} \|^2 \\
\leq & \frac{n \gamma \Delta t}{1-\gamma}  \| g_{w_1}(x,t^1) - g_{w_2}(x, t^1) \|^2 
+  \frac{2 \gamma^2  \Delta t } {1 - \gamma} \sum_{ind=1}^{n} \sum_{k=1}^{ind} \| g_{w_1}(x,t^k) - g_{w_2}(x,t^k) \|^2 
+ \frac{\gamma}{\Delta t} \sum_i \| r_i^{1} \|^2  \\
\leq & \frac{(n+1) \gamma \Delta t}{1-\gamma}  \| g_{w_1}(x,t^1) - g_{w_2}(x, t^1) \|^2 
+ \frac{2\gamma^2  \Delta t } {1 - \gamma} \sum_{ind=1}^{n} \sum_{k=1}^{ind} \| g_{w_1}(x,t^k) - g_{w_2}(x,t^k) \|^2,
\end{align*}
where we apply (\ref{stability1}) in the last inequality.

Therefore, 
\begin{align*}
 \sum_i \| r_i^{n+1} \|^2 
\leq &\frac{2(n+1) (\Delta t)^2}{1-\gamma}  \| g_{w_1}(x,t^1) - g_{w_2}(x, t^1) \|^2
+  \frac{4\gamma  (\Delta t)^2 } {1 - \gamma} \sum_{ind=1}^{n} \sum_{k=1}^{ind} \| g_{w_1}(x,t^k) - g_{w_2}(x,t^k) \|^2 \\
\leq & \frac{2T(\Delta t)}{1-\gamma}  \| g_{w_1}(x,t^1) - g_{w_2}(x, t^1) \| ^2
+  \frac{4\gamma T^2}{1-\gamma} \max_{1\leq k \leq N-1} \| g_{w_1}(x,t^k) - g_{w_2}(x,t^k) \|^2,
\end{align*}
where $T = N \Delta t$ is the final time. 
Also, from (\ref{ineq-1}), we have
\begin{align*}
\|r^{n+1}\|_a^2 
\leq& \frac{\gamma\Delta t}{1-\gamma}  \| g_{w_1}(x,t^1) - g_{w_2}(x, t^1) \|^2 
+  \frac{2\gamma^2 \Delta t} {1-\gamma} \sum_{k=1}^n \| g_{w_1}(x,t^k) - g_{w_2}(x,t^k) \|^2 \\
\leq& \frac{\gamma \Delta t}{1-\gamma}  \| g_{w_1}(x,t^1) - g_{w_2}(x, t^1) \|^2 
+  \frac{2\gamma^2 T } {1-\gamma} \max_{1\leq k \leq N-1} \| g_{w_1}(x,t^k) - g_{w_2}(x,t^k) \|^2. 
\end{align*}

With these two inequalities, we can obtain the conclusion.

\end{proof}

\bibliographystyle{abbrv}
\bibliography{references,references4,references1,references2,references3,decSol,refParamLearn}

\end{document}